\documentclass[12pt,reqno]{amsart}
\usepackage{amssymb}
\numberwithin{equation}{section}

\newtheorem{theorem}{Theorem}[section]
\newtheorem{lemma}[theorem]{Lemma}

\theoremstyle{remark}

\newcommand{\R}{\mathbb{R}}

\newcommand{\N}{\mathbb{N}}

\newcommand{\cL}{\mathcal{L}}
\newcommand{\cT}{\mathcal{T}}
\newcommand{\K}{\mathcal{K}}

\newcommand{\U}{\underline}

\newcommand{\al}{\alpha}
\newcommand{\be}{\beta}

\newcommand{\De}{\Delta}

\newcommand{\ka}{\kappa}

\newcommand{\fy}{\varphi}
\newcommand{\e}{\varepsilon}
\newcommand{\p}{\partial}
\newcommand{\la}{\lambda}
\newcommand{\de}{\delta}

\newcommand{\ta}{\tau}
\newcommand{\s}{\sigma}

\newcommand{\x}{\xi}

\newcommand{\na}{\nabla}
\newcommand{\I}{\infty}
\newcommand{\LR}[1]{\langle {#1} \rangle}
\newcommand{\lec}{{\, \lesssim \, }}
\newcommand{\gec}{{\, \gtrsim \, }}

\newcommand{\br}{\overline}
\newcommand{\ti}{\widetilde}

\newcommand{\cS}{\mathcal{S}}

\newcommand{\supp}{\operatorname{supp}}

\newcommand{\EQ}[1]{\begin{equation} \begin{split} #1
 \end{split} \end{equation}}
\newcommand{\pr}{\\ &}
\newcommand{\pt}{&}
\newcommand{\pq}{\quad }

\newcommand{\CAS}[1]{\begin{cases} #1 \end{cases}}

\newcommand{\IN}[1]{\text{ in }#1}

\begin{document}

\title[Threshold solutions in the case of mass-shift]{Threshold solutions \\ 
in the case of mass-shift \\ for the critical Kline-Gordon equation}

\author{Slim Ibrahim}
\address{Department of Mathematics and Statistics \\ University of Victoria \\
 PO Box 3060 STN CSC \\ Victoria, BC, V8P 5C3\\ Canada}
\email{ibrahim@math.uvic.ca}
\thanks{S. I.  is partially supported by NSERC\# 371637-2009 grant and a start up fund from University of Victoria.}

\author{Nader Masmoudi}
\address{New York University \\ The Courant Institute for Mathematical Sciences}
\email{masmoudi@courant.nyu.edu}
\thanks{N. M. is partially supported by an NSF Grant DMS-0703145}

\author{Kenji Nakanishi}
\address{Department of Mathematics, Kyoto University}
\email{n-kenji@math.kyoto-u.ac.jp}


\begin{abstract}
We study global dynamics for the focusing nonlinear Klein-Gordon equation with the {\it energy-critical} nonlinearity in two or higher dimensions when the energy equals the threshold given by the ground state of a {\it mass-shifted} equation, and prove that the solutions are divided into scattering and blowup. In short, the Kenig-Merle scattering/blowup dichotomy \cite{KM1,KM2} extends to the threshold energy in the case of mass-shift for the critical nonlinear Klein-Gordon equation. 
\end{abstract}

\maketitle 

\tableofcontents 

\section{Introduction}
In this paper, we continue from \cite{ScatBlow} the study of global dynamics for the nonlinear Klein-Gordon equation (NLKG) 
\EQ{ \label{NLKG}
 \ddot u - \De u + u = f'(u), \pq u(t,x):\R^{1+d}\to\R,}
with the focusing energy-critical nonlinearity $f(u)$ in two or higher dimensions 
\EQ{ \label{f}
 \CAS{f(u)\sim \la\frac{e^{\al|u|^2}-1-\al|u|^2-\frac{\al^2}{2}|u|^4}{1+|u|^\be} &(d=2) \\ f(u)=|u|^{2^\star}/2^\star, \pq 2^\star:=2d/(d-2) &(d\ge 3)}}
where $\al>0$, $\be\ge 2$, and $\la>0$ are given constants. See \S\ref{def exp} for more precise assumption for $d=2$. 
We are interested in the cases where NLKG \eqref{NLKG} {\it does not} have the ground state, but the {\it mass-shifted} equation 
\EQ{ \label{MS-NLKG}
 \ddot u - \De u + cu = f'(u)}
with some $c\in(0,1)$ for $d=2$ and $c=0$ for $d\ge 3$ does have the ground state $Q$. Here ``ground state" refers to a positive solution of
\EQ{ \label{MS-stat}
 -\De Q + cQ = f'(Q)}
with the least energy among all the static solutions. The conserved energy for \eqref{MS-NLKG} is denoted by 
\EQ{
 E^{(c)}(u):=\int_{\R^d}\frac{|\dot u|^2+|\na u|^2+c|u|^2}{2}-f(u)dx,}
while we omit $(c)$ when $c=1$, i.e.~for \eqref{NLKG}. 

Kenig and Merle \cite{KM1,KM2} proved that all solutions with energy below the ground state for the nonlinear wave equation (NLW) and the nonlinear Schr\"odinger equation (NLS) with the critical power $f(u)=|u|^{2^\star}$ ($d\ge 3$) are divided into the scattering and the blowup, distinguished by some explicit variational conditions. 
We extended it in \cite{ScatBlow} to NLKG including the subcritical and the critical cases, where the distinction was given by the sign of scaling derivatives of the static energy. 

A new feature in \cite{ScatBlow} is that the ground state for the critical NLKG may be absent, but the Kenig-Merle dichotomy is still valid below the ground state $Q$ of the {\it mass-shifted} equation \eqref{MS-NLKG}. 
The phenomenon of mass-shift in two dimensions has been studied in \cite{ScatBlow,TM}, in view of the Trudinger-Moser inequality of the form
\EQ{
 \|\na \fy\|_2 \le \ka \implies \int_{\R^2}f(u)dx \le c\|u\|_2^2/2,}
where $\|\cdot\|_p$ denotes the $L^p(\R^d)$ norm. In particular, necessity of the negative power $\be\ge 2$ in \eqref{f} was revealed in \cite{TM}. 

Since the mass-shifted ground state $Q$ does not solve the NLKG \eqref{NLKG} without the mass-shift, it cannot be a definite obstruction for the scattering/blowup dichotomy, unlike the cases of NLW/NLS where the ground state is a true solution which neither scatters or blows up. Therefore it is natural to ask what happens on the threshold energy, expecting different dynamical pictures from NLW/NLS. 

In the cases of NLW/NLS with the threshold energy, Duyckaerts and Merle \cite{DM1,DM2} proved that there are essentially three new solutions only, in addition to the scattering/blowup dichotomy: $Q$ and $W_\pm$, where $W_\pm$ converge exponentially to $Q$ as $t\to\I$, while $W_+$ blows up in $t<0$ and $W_-$ scatters as $t\to-\I$. See also \cite{DKM} for a revised proof of the global existence. 

We will prove that the dynamics in the mass-shifted case is simpler, where the Kenig-Merle picture essentially extends. The blowup part is a small modification of the Payne-Sattinger argument \cite{PS} with variational estimates on the mass-shifted threshold. However, the scattering part seems highly nontrivial and the proof is quite different between the two dimensional case and the higher dimensional case. 

The main difficulty in the 2D case is that as the solution concentrates the kinetic energy $\|\na u\|_2$, we might have blowup for the Strichartz estimates on the nonlinearity, which typically take the form 
\EQ{
 \|f'(u)\|_{L^1_tL^2_x} \lec C(\|u\|_S),}
for some Strichartz space-time norm $\|u\|_S$ for the Klein-Gordon equation. Note that such an estimate is much stronger than the energy bound $f(u)\in L^1_x$, especially for the exponential nonlinearity. Hence we need first to preclude concentration of kinetic energy before starting any Strichartz-type analysis, including the Kenig-Merle approach and the Bahouri-G\'erard profile decomposition. What saves us is the fact that the variational lower bound on the scaling derivative of static energy 
\EQ{
 K_2(u):=\int_{\R^d}[2|\na u|^2-d(uf'-2f)(u)]dx \gec \|\na u\|_2^2}
does not degenerate at the threshold in the case of mass-shift. This might look surprising in view of the Trudinger-Moser inequalities breaking down beyond the threshold, but it is not so strange if one takes account of the fact that the mass-shifted ground state has strictly bigger energy 
\EQ{
 E(Q) = E^{(c)}(Q) + (1-c)\|Q\|_2^2}
than the threshold $E^{(c)}(Q)$. Thus our analysis in the 2D case is mostly variational, reducing to the arguments below the threshold \cite{ScatBlow} or in the defocusing case \cite{2DcNLKG}. 

In the higher dimensional case, the Strichartz estimate causes no problem, but all the variational estimates do degenerate as the energy concentrates. Hence we use the profile decomposition to reduce the concentration problem to the scaling limit, i.e.~the critical NLW on the threshold, which was completely solved by Duyckaerts and Merle \cite{DM2}. Relying crucially on their classification, our main task then is to preclude concentrating convergence to the NLW ground state, either in finite time or at the time infinity. It is carried out by exploiting the fact that the mass term of NLKG induces time oscillation especially when the energy is concentrating. It is noteworthy that here we use the best constant in a variational estimate on the scaling derivative of static energy. 

In order to state our main result, we need to introduce some notation as well as the precise assumptions on $f$ in two dimensions, which is the same as in \cite{ScatBlow}. Let $F$ be the nonlinear part of energy 
\EQ{
 F(\fy):=\int_{\R^d}f(\fy(x))dx.}
Let $D$ be the scaling derivative acting both on functions and on functionals 
\EQ{
 Df(u):=uf'(u),\pq DF(\fy):=\int_{\R^d}Df(\fy(x))dx.}

\subsection{Exponential nonlinearity in two dimensions}\label{def exp}
First we assume that $1$ is the true mass coefficient in NLKG \eqref{NLKG}. Namely $f:\R\to\R$ is $C^2$ satisfying 
\EQ{ \label{f 2D-0}
 f\in C^2(\R;\R), \pq f(0)=f'(0)=f''(0)=0.}
For the variational argument, we need some monotonicity and convexity, 
\EQ{ \label{f 2D-1}
 \exists p>4,,\pq (D-p)f \ge 0, \ (D-2)(D-p)f \ge 0,}
and for the global Strichartz estimate, we need some decay at $u=0$: 
\EQ{ \label{f 2D-2}
 \exists p>4,\pq \limsup_{|u|\to 0}|u|^{-p}|D^2f(u)|<\I.}
In the last condition, $D^2f$ can be replaced with $|u|^2f''(u)$. They are satisfied by focusing powers $f(u)=\la|u|^p$ for $p>4$ and $\la>0$, as well as their sum or series. 

The exponential behavior is described for large $|u|$ by the following assumptions. 
\EQ{ \label{f 2D-3}
 \lim_{|u|\to\I}\frac{Df(u)}{f(u)}=\I,\pq \exists\ka_0>0,\ \CAS{\forall\ka>\ka_0,  &\lim_{|u|\to\I}e^{-\ka|u|^2}f''(u)=0, \\ \forall\ka<\ka_0, &\lim_{|u|\to\I}e^{-\ka|u|^2}f(u)=\I.}}
It is easy to see that they are satisfied if 
$f(u)=e^{\ka_0|u|^2}g(u)$, with $g(u)$ satisfying $\limsup_{|u|\to\I}|u|^{-p}|g''(u)|<\I$ and $\liminf_{|u|\to\I}|u|^p g(u)>0$ for some $p>0$. 
The most crucial assumption for the mass-shift is the following
\EQ{ \label{f 2D-4}
 c:=\sup\{2F(\fy)\|\fy\|_2^{-2}\mid 0\not=\fy\in H^1(\R^2),\ \ka_0\|\na\fy\|_2^2\le 4\pi\}<1.}
It was shown in \cite{TM} that the necessary and sufficient condition for $c<\I$ is 
\EQ{ \label{TM cond}
 \limsup_{|u|\to 0}|u|^{-2}f(u)<\I\ \text{ and }\ \limsup_{|u|\to\I}e^{-\ka_0|u|^2}|u|^2f(u)<\I.}
Hence the condition $c<1$ is satisfied by $\la f(u)$ for sufficiently small $\la>0$. Finally, we need the following assumption 
\EQ{ \label{bound on Df}
 \limsup_{|u|\to\I} e^{-\ka_0|u|^2}Df(u)<\I,}
which is not so far from the second of \eqref{TM cond}. Indeed, \eqref{bound on Df} is sufficient for the latter, while $\liminf_{|u|\to\I} e^{-\ka_0|u|^2}Df(u)<\I$ is necessary. 

\subsection{Ground state and the scaling derivative of energy}
For the above $f$ i.e.~either \eqref{f} with $d\ge 3$ or \eqref{f 2D-0}--\eqref{f 2D-4} with $d=2$, there is the ground state as the mini-max critical point for the static energy 
\EQ{
 J^{(c)}(\fy):=\int_{\R^d}\frac{|\na\fy|^2+|\fy|^2}{2}-f(\fy) dx}
for the scaling family $\fy_{\al,\be}^\la:=e^{\al\la}\fy(x/e^{\be\la})$, for any $(\al,\be)\in\R^2$ satisfying 
\EQ{ \label{range albe}
 \al\ge 0,\pq 2\al+d\be\ge 0,\ 2\al+(d-2)\be\ge 0,\ (\al,\be)\not=(0,0).}
More precisely, there is a radial $Q\in H^2(\R^d)$ solving the mass-shifted static equation \eqref{MS-stat}, as well as the constrained minimization 
\EQ{ \label{def m}
 J^{(c)}(Q)=m^{(c)} 
 :\pt= \inf\{J^{(c)}(\fy) \mid 0\not=\fy\in H^1(\R^d),\ K_{\al,\be}^{(c)}(\fy)=0\},}
where $K_{\al,\be}^{(c)}$ is the scaling derivative 
\EQ{
 K_{\al,\be}^{(c)}(\fy)\pt:=\cL_{\al,\be}J^{(c)}(\fy):=\p_\la|_{\la=0}J^{(c)}(\fy_{\al,\be}^\la)
 \pr=\LR{-\De\fy+c\fy-f'(\fy)|(\al-\be x\cdot\na)\fy}
 \pr=\int_{\R^d}\left[\al+\be\frac{d-2}{2}\right]|\na\fy|^2+\left[\al+\be\frac{d}{2}\right]|\fy|^2-(\al D+\be d)f(\fy)dx.}
See \cite{ScatBlow} for a proof of the existence of $Q$, where it was also proved that for any $a>c$
\EQ{
 m^{(a)}=m^{(c)}, }
but not achieved by $J^{(a)}(\fy)$. Henceforth we will omit $(c)$, i.e.~$m:=m^{(c)}$. 

We will mainly use the following $K$'s
\EQ{
 \pt K_0^{(c)}(\fy):=K_{1,0}^{(c)}(\fy)=\|\na\fy\|_2^2+c\|\fy\|_2^2-G_0(\fy), \pq G_0:=DF, 
 \pr K_\I^{(c)}(\fy):=K_{0,1}^{(c)}(\fy)=\frac{d-2}{2}\|\na\fy\|_2^2+\frac{c}{2}\|\fy\|_2^2-F(\fy), 
 \pr K_2^{(c)}(\fy):=K_{d,-2}^{(c)}(\fy)=2\|\na\fy\|_2^2-G_2(\fy), \pq G_2:=d(D-2)F,}
where the index $p$ of $K_p$ refers to the scaling which preserves $\|\fy_{\al,\be}^\la\|_{L^p(\R^d)}$. 

\subsection{Main result} 
Local wellposedness in the energy space $(u,\dot u)\in H^1(\R^d)\times L^2(\R^d)$ is known in the critical case, both for $d\ge 3$ \cite{GV} and for $d=2$ \cite{IMM2}, though the existence time is not uniformly bounded in terms of the norm. Hence we can extend any local solution uniquely to the maximal existence interval in both positive and negative time directions, where the solution is continuous in the energy space. Blow-up is defined by that the solution cannot be extended continuously beyond some finite time. Scattering for $t\to\I$ means that there is a solution $v$ of the free Klein-Gordon equation such that $u-v\to 0$ as $t\to\I$ in the energy space. The scattering for $t\to-\I$ is defined similarly. 

\begin{theorem}
Let $d\ge 3$ with \eqref{f} or $d=2$ with \eqref{f 2D-0}--\eqref{f 2D-4} and \eqref{bound on Df}. If $d=2$ then let $c$ be given by \eqref{f 2D-4}, otherwise let $c=0$. Define $m>0$ by \eqref{def m}. Then for any solution in the energy space $(u,\dot u)\in H^1(\R^d)\times L^2(\R^d)$ with $E(u)=m$ 
satisfies one of the following two. Let $I$ be the maximal existence interval. 
\begin{enumerate}
\item $I=\R$ and $u$ scatters both for $t\to\I$ and for $t\to-\I$, and $K_{\al,\be}(u)\ge 0$ for all $t\in\R$ and for all $(\al,\be)$ in \eqref{range albe}. 
\item $I$ is bounded, i.e.~$u$ blows up both in positive and negative times, and $K_{\al,\be}(u)<0$ for all $t\in I$ and for all $(\al,\be)$ in \eqref{range albe}. 
\end{enumerate}
\end{theorem}

\subsection{Notation and some preliminary estimates} 
For any $a\in[c,1)$, let 
\EQ{
 \pt \K_{\al,\be}^{+(a)}:=\{(u,\dot u)\in H^1\times L^2 \mid E^{(a)}(u)<m,\ K_{\al,\be}^{(a)}(u)\ge 0\},
 \pr \K_{\al,\be}^{-(a)}:=\{(u,\dot u)\in H^1\times L^2 \mid E^{(a)}(u)<m,\ K_{\al,\be}^{(a)}(u)< 0\},}
where $u$ and $\dot u$ are just arbitrary functions respectively from $H^1$ and $L^2$ satisfying the above conditions (there is no actual time here). 
It was shown in \cite{ScatBlow} that both sets are open and connected (note that $m=m^{(a)}$), and independent of $(\al,\be)$ in the range \eqref{range albe}. So we will omit the subscript $\al,\be$. Since $E^{(a)}$ is increasing for $a$, we deduce that $\K^{+(a)} \subset \K^{+(c)}$, while  $\K^{-(a)} \subset \K^{-(c)}$ is trivial. Taking the limit $a\to 1-0$, we conclude 
\EQ{
 \pt \K^+:=\{(u,\dot u)\in H^1\times L^2 \mid E(u)\le m,\ K_{\al,\be}(u)\ge 0\} \subset \K^{+(c)},
 \pr \K^-:=\{(u,\dot u)\in H^1\times L^2 \mid E(u)\le m,\ K_{\al,\be}(u)< 0\} \subset \K^{-(c)},}
which are also independent of $(\al,\be)$. Let 
\EQ{
 M(t) := \|\dot u\|_2^2+(1-c)\|u\|_2^2.}
Then $E(u)=m$ implies that 
\EQ{
  J^{(c)}(u(t))=m-M(t)/2.}
Lemma 2.12 in \cite{ScatBlow} with the mass $c$ implies that for $(\al,\be)$ in \eqref{range albe} and $(d,\al)\not=(2,0)$, there exists $\de=\de_{\al,\be}>0$ such that 
\EQ{ \label{bd K}
 \pt (u,\dot u)\in \K^{+(c)} \implies K_{\al,\be}^{(c)}(u) \ge \min(\de K^{(c)Q}_{\al,\be}(u),\frac{\br\mu}{2}M(t)), 
 \pr (u,\dot u)\in \K^{-(c)} \implies K_{\al,\be}^{(c)}(u) \le -\frac{\br\mu}{2}M(t),}
where $K_{\al,\be}^{(c)Q}$ denotes the free version (without the nonlinear term) of $K_{\al,\be}^{(c)}$, and 
\EQ{
 \br\mu:=2\al+\max(\be d,\be(d-2)).}
Both in $\K^+$ and in $\K^-$, the key identity of our argument will be 
\EQ{ \label{eq times u}
 \ddot y/2 = \|\dot u\|_2^2 - K_0(u),}
where $y(t):=\|u\|_2^2$. The energy density is denoted by 
\EQ{
 e_F^{(c)}:=\frac{|\dot u|^2+|\na u|^2+c|u|^2}{2}, \pq e_N^{(c)}:=e_F^{(c)}-f(u).}

For any space-time function $u(t,x)$, we use the vector notation 
\EQ{
 \vec u:=\LR{\na}u-i\dot u,\pq \LR{\na}:=\sqrt{1-\De},}
then $\|\vec u\|_2^2=\int 2e_Fdx$. 
Since the energy is not sign definite, the finite propagation speed is not obvious from the energy conservation, but we have
\begin{lemma}[propagation of exterior smallness] \label{ext small}
Let $u$ be a solution of NLKG \eqref{NLKG} around $t=T$ satisfying for some $R>0$, $\al\in\R^d$ and $\e>0$, 
\EQ{
 \int_{|x-\al|>R}e_F(T)dx<\e.}
There is a constant $\e_0>0$ determined by the equation, such that if $\e\le\e_0$ then $u$ can be extended to the exterior cone $|x-\al|>R+|t-T|$, and for all $t\in\R$, 
\EQ{
 \int_{|x-\al|>R+|t-T|}e_F(t)dx \lec \e.}
\end{lemma}
This lemma holds as long as NLKG is locally wellposed in the energy space, the energy is conserved, and $f(u)$ is superquadratic. 
\begin{proof}
Using the extension from $|x-\al|>R$ into the inside, we can find $\fy\in H^1$ and $\psi\in L^2$ such that $\fy(x)=u(T,x)$ and $\psi(x)=\dot u(T,x)$ for $|x-\al|>R$ and 
\(
 \|\fy\|_{H^1}^2+\|\psi\|_2^2\lec \e. \)
If $\e>0$ is small enough, then there is a global solution $v$ with $(v(T),\dot v(T))=(\fy,\psi)$ satisfying $E(v)\sim E_F(v)\lec\e$ for all $t\in\R$. Here the global existence comes from the a priori small bound on the energy, not by the scattering. By the finite propagation property of the linear equation, $v(t,x)=u(t,x)$ for $|x-\al|>R+|t-T|$, which implies the conclusion. 
\end{proof}

\section{Blow-up in $\K^-$} 
First we choose $p=2+\e\in(2,4)$ so that 
\EQ{
 (D-4/(2-\e))f \ge 0, \pq (D-2)^2 f \ge 0,}
which means that $f$ is slightly super-quadratic, and they are weaker conditions than \eqref{f 2D-1}. Define a new functional $H^{(c)}_p$ by 
\EQ{
 H^{(c)}_p(\fy) = J(\fy)-K_0(\fy)/p.}
Then for any $\fy\in H^1$ satisfying $K^{(c)}_0(\fy)< 0$, we have 
\EQ{
 m=J^{(c)}(Q)=H_p^{(c)}(Q) \le H_p^{(c)}(\fy).}
\begin{proof}
Consider the scaling family $\fy^\la:=\fy^\la_{1,0}$. Then 
\EQ{
 \pt\p_\la H_p^{(c)}(\fy^\la) = (\cL_{1,0}-\frac{1}{p}\cL_{1,0}^2)J^{(c)}(\fy^\la)
 \pr= \frac{\e}{4}K_0^{(c)Q}(\fy^\la) + \frac{4-p}{p}(D-\frac{4}{4-p})F(\fy^\la)+\frac{1}{p}(D-2)^2F(\fy^\la) \ge 0,}
which can be checked directly or by using 
\EQ{
 (\cL_{1,0}-2)^2 J = -(\cL_{1,0}-2)^2 F,}
noting that $(\al,\be)=(1,0)$ implies that $\br\mu=2=\U\mu:=2\al+\min(\be d,\be(d-2))$. 
Since $K^{(c)}_0(\fy)<0$, there exists $\la<0$ such that $K^{(c)}_0(\fy^\la)=0$. Since $Q$ is a minimizer of $J^{(c)}$ over $K^{(c)}=0$, and $H_p^{(c)}(\fy^\la)$ is non-decreasing in $\la$, we get 
\EQ{
 H_p^{(c)}(Q)=J^{(c)}(Q) \le J^{(c)}(\fy^\la) = H_p^{(c)}(\fy^\la) \le H_p^{(c)}(\fy).}
\end{proof}

Now assume by contradiction that the solution $u$ extends to $t\to\I$. 
We can rewrite \eqref{eq times u} by using the fact that $E(u)=m=H_p^{(c)}(Q)$, 
\EQ{ \label{DI y}
 \ddot y \pt= (2+p) \|\dot u\|_{L^2}^2 + 2p(H^{(1)}_p(u)-m)
 \pr= (4+\e)\|\dot u\|_{L^2}^2 + (1-c)\e\|u\|_{L^2}^2 + 2p(H^{(c)}_p(u)-H^{(c)}_p(Q))
 \pr\ge (1+\e/4)\dot y^2/y + (1-c)\e y.}
Let $z:=y^{-\e/4}$. Using the last inequality, we get 
\EQ{ \label{DI z}
 \ddot z = -\e z[\ddot y-(1+\e/4)\dot y^2/y]/(4y) \le -(1-c)\e^2z/4.}
Then by Sturm-Liouville, $z(t)$ has to cross $0$ within any time interval longer than $2\pi/(\e\sqrt{1-c})$, a contradiction. Hence $u$ cannot be a global solution. 

\section{Scattering in $\K^+$ in two dimensions}
\subsection{Global wellposedness in 2D}
Under the assumption \eqref{bound on Df}, we can use the Trudinger-Moser inequality on the disk: For any $\fy\in H^1(\R^2)$, 
\EQ{ \label{TMD}
 \supp\fy\subset\{|x|<R\},\ \|\na\fy\|_2^2\le 2m \implies \int_{|x|<R}e^{\ka_0|\fy|^2}dx \lec R^2.}
Suppose that the solution $u$ in $\K^+$ is defined for $-T<t<0$ with $E(u)=m$, blowing up as $t\to-0$. Since 
\EQ{ \label{deri bd}
 \|\na u(t)\|_2^2+\|\dot u(t)\|_2^2=2[E(u)-K_\I(u(t))] \le 2E(u)=2m,}
$u$ stays in the (sub)critical region $\|\na u(t)\|_2^2\le 2m$, and so it can blow up only if 
\EQ{
 \limsup_{t\to-0}\|\na u(t)\|_2^2=2m.}
Using the finite propagation together with the uniform local wellposedness in the subcritical range $\|\na u(t)\|_2<2m$ (see \cite{IMM1}), we deduce that for some $x^*\in\R^2$  
\EQ{
 \limsup_{t\to-0}\int_{|x-x^*|<|t|}|\na u|^2dx = 2m,}
which, together with \eqref{deri bd}, implies 
\EQ{
 \liminf_{t\to-0}\int_{|x-x^*|>|t|}[|\na u|^2+|\dot u|^2]dx=0.}
Since $\dot u(t)$ is bounded in $L^2_x$, $u(t)$ converges strongly in $L^2_x$, while the above identities imply that $\na u(t)\to 0$ in $\cS'_x$. Hence $\|u(t)\|_2\to 0$. Thus we get $t_n\nearrow 0$ such that 
\EQ{
 \int_{|x-x^*|>|t_n|}e_F(t_n)dx \to 0.}
Applying Lemma \ref{ext small} to each $t_n$, we deduce that 
\EQ{
 \supp u(t,x) \subset\{|x-x^*|\le -t\}.}

So the Trudinger-Moser \eqref{TMD} on the disk $|x|<|t|$ implies 
\EQ{
 G_0(u)+F(u)+\|u\|_2^2\lec t^2,}
and since $\|\na u\|_2^2/2\to E(u)$ as $t\to-0$, we have also 
$\|\dot u\|_2^2 \to 0$. 
But these are contradicting the energy equipartition.  Indeed, 
\EQ{
 \p_t\LR{u|\dot u}=\|\dot u\|_2^2-\|\na u\|_2^2-\|u\|_2^2-G_0(u)\to-2m<0,}
cannot hold with $|\LR{u|\dot u}|\le\|u\|_2\|\dot u\|_2=o(t)$ as $t\to-0$. 

\subsection{Scattering in two dimensions}
We use a uniform variational estimate: 
\begin{lemma} \label{K2 unif bd}
In the case of mass shift $c\in(0,1)$, for any function $\fy\in H^1(\R^2)$ with $J(\fy)\le m$ and $K_2(\fy)\ge 0$ we have 
\EQ{
 K_2(\fy)\ge (1-c)\|\na\fy\|_2^2.}
\end{lemma}
\begin{proof}
We may assume $\fy\not=0$. Let $\ti c:=(1+c)/2$, then 
\EQ{
 J^{(\ti c)}(\fy)=\|\na\fy\|_2^2/2-K_\I^{(\ti c)}(\fy)\le m-(1-c)\|\fy\|_2^2/4.}
For any $R>1$, let $\fy_R:=\fy(x/R)$. Then 
\EQ{
 J^{(\ti c)}(\fy_R)=\|\na\fy\|_2^2/2-R^2K_\I^{(\ti c)}(\fy)<m}
for any $R\in[1,R_*)$, where $R_*=R_*(\fy)$ is determined by 
\EQ{
  (R_*^2-1)K_\I^{(\ti c)}(\fy)=(1-c)\|\fy\|_2^2/4.}
Substituting 
\( 
 K_\I^{(\ti c)}(\fy)\le\ti c\|\fy\|_2^2/2=(1+c)\|\fy\|_2^2/4,
\)
we obtain 
\EQ{
 R_*^2 \ge 2/(1+c).}
By continuity in $R$, $\fy_R$ stays in $\K^+$ for $1\le R\le R_*$, so 
\EQ{
 0\le K_2(\fy_{R_*})=2\|\na \fy\|_2^2-2R_*^2G_2(\fy).}
Therefore 
\EQ{ 
 K_2(\fy)=2\|\na \fy\|_2^2-2G_2(\fy) \ge 2(1-R_*^{-2})\|\na \fy\|_2^2 \ge (1-c)\|\na\fy\|_2^2.}
\end{proof}

\subsubsection{Case 1: Vanishing kinetic energy} Suppose that for any $L\ge 2$ we have
\EQ{ \label{vanish mean kine}
 \liminf_{T\to\I}\int_{T}^{T+L}\|\na u(t)\|_2^2dt = 0.}
\begin{lemma}
If $u$ is a global solution in $\K^+$ with $E(u)\le m$ and \eqref{vanish mean kine}, then for each $L\ge 2$ there exists a sequence $T_n\nearrow\I$ such that 
\EQ{
 \pt \int_{T_n}^{T_n+L}\|\na u(t)\|_2^2dt\to 0,
 \pq \int_{T_n}^{T_n+L}G_0(u)+F(u)dt\to 0,
 \pr \int_{T_n}^{T_n+L}\|\dot u(t)\|_2^2dt \to LE(u), 
 \pq \int_{T_n}^{T_n+L}\|u(t)\|_2^2dt \to LE(u).}
\end{lemma}
\begin{proof}
The first two are obvious from \eqref{vanish mean kine}, $K_2(u(t))\ge 0$ and $(D-2)f\gec f$. The latter two are essentially the energy equipartition, which follows from 
\EQ{
 \p_t\LR{\dot u|u}=\|\dot u\|_2^2-K_0(u)=\|\dot u\|_2^2-\|u\|_2^2+2F(u)-K_2(u)/2.}By \eqref{vanish mean kine}, we can find $S_n\nearrow\I$ such that 
\EQ{
 \int_{S_n}^{S_n+nL}[\|\na u\|_2^2+G_0(u)+F(u)]dt \to 0,}
and $T_n\in(S_n,S_n+nL)$ such that 
\EQ{
 |[\LR{\dot u|u}]_{T_n}^{T_n+L}| \lec E(u)/n \to 0,}
and hence 
\EQ{
 \int_{T_n}^{T_n+L}[\|\dot u\|_2^2-\|u\|_2^2]dt \to 0, \pq \int_{T_n}^{T_n+L}[\|\dot u\|_2^2+\|u\|_2^2]dt \to 2LE(u).}
\end{proof}

For each $n\in\N$, let $I_n:=(T_n,T_n+L)$ and let 
\EQ{
 A_n:=\{t\in I_n \mid 2\|u\|_2^2\le E(u)\},}
then we have, as $n\to\I$, 
\EQ{
 2|A_n|E(u)\pt=\int_{A_n}[2\|u\|_2^2+\|\dot u\|_2^2-\|u\|_2^2+\|\na u\|_2^2-2F(u)]dt
 \pr\le |A_n|E(u)+o(1),}
and so $|A_n|\to 0$. Since the uniform bound on $\|\dot u(t)\|_2$ implies the uniform continuity of $\|u(t)\|_2$, we deduce that 
\EQ{
 \liminf_{n\to\I}\inf_{t\in I_n}\|u(t)\|_2^2 \ge E(u)/4,}
and therefore, using $2F(u)\le c\|u\|_2^2$ as well, 
\EQ{
 \limsup_{n\to\I}\sup_{t\in I_n}[\|\dot u(t)\|_2^2+\|\na u(t)\|_2^2] \le \left[2-(1-c)/4\right]E(u),}
which implies that $u$ is uniformly subcritical on $I_n$ (for large $n$), for the Strichartz estimate. In particular, there is some $\de=\de(c)>0$ such that for any interval $I\subset I_n$, 
\EQ{
 \|f'(u)\|_{L^1_tL^2_x(I)} \lec \|u\|_{L^8_tL^{16}_x(I)}^{8\de}\|u\|_{L^4_t(B_{\I,2}^{1/4}\cap B^{3/8}_{8,2})}^{4(1-\de)},}
where the last norm is uniformly bounded thanks to the Strichartz estimate (see Section 3.1 and in particular Corollary 3.2 of \cite{2DcNLKG}, but we do not need the scattering since $L$ is fixed). 
Let $v_n$ be the free solution with $\vec v_n(T_n)=\vec u(T_n)$. Then by the energy and interpolation inequalities, 
\EQ{
 \|\vec v_n-\vec u\|_{L^\I_tL^2_x(I_n)} \lec \|u\|_{L^8_tL^{16}_x(I_n)}^{8\de}
 \pt\lec \|\na u\|_{L^2_{t,x}(I_n)}^{2\de}\|\na u\|_{L^\I_tL^2_x(I_n)}^{5\de}\|u\|_{L^\I_tL^2_x(I_n)}^{\de} 
 \pr\lec \|\na u\|_{L^2_{t,x}(I_n)}^{2\de}\to 0,}
and so 
\( 
 \|\na v_n\|_{L^2_{t,x}(I_n)} \to 0. \)
Decompose $v_n$ in the time phase by
\EQ{
 v_n(T_n+t)=e^{it\LR{\na}}\fy_n^+ + e^{-it\LR{\na}}\fy_n^-, \pq \fy_n^\pm\in H^1,}
then the mean kinetic energy is computed in the Fourier space 
\EQ{
 \pt\int_0^L\||\x|[e^{it\LR{\x}}\hat\fy_n^++e^{-it\LR{\x}}\hat\fy_n^-]\|_{L^2_\x}^2dt
 \pr=L\||\x|\hat\fy_n^+\|_2^2 + L\||\x|\hat\fy_n^-\|_2^2+\Im\LR{\LR{\x}^{-1}(e^{2iL\LR{\x}}-1)|\x|\hat\fy_n^+||\x|\hat\fy_n^-},}
which is equivalent to $L[\|\na\fy_n^+\|_2^2+\|\na\fy_n^-\|_2^2]$ since $L\ge 2$. Thus we conclude that 
\EQ{
 \|\na v_n\|_{L^\I_t(\R;L^2_x)} \to 0 \pq (n\to\I),}
which allows us to resort to the small data scattering from $t=T_n$ for large $n$. Indeed, since the kinetic energy is uniformly small for the free solution $v_n$, we can perform the iteration argument in function spaces where exponential nonlinearity $f(u)\sim e^{\al|u|^2}$ could be controlled for any $\al>0$ by the Strichartz estimate.

Although it is the desired conclusion, it contradicts the assumption \eqref{vanish mean kine}. Hence we have only the following case. 

\subsubsection{Case 2: Growing virial with zero momentum} Now we may assume that for some $L\ge 2$  and $\de>0$, 
\EQ{ \label{kine lowbd}
 \inf_{T>0}\int_{T}^{T+L}\|\na u(t)\|_2^2dt \ge \de.}
In addition to that, we assume that the conserved total momentum is zero
\EQ{
 P(u):=\LR{\dot u|\na u}=0.}
Otherwise it will be reduced to this by the Lorentz transform in the next section. Now we introduce the {\it concentration radius} $R(t)$ in a way similar to \cite{2DcNLKG}. Noting that 
\EQ{
 E_F^{(1-c)}(u) = E(u)-K_\I^{(c)}(u) \le E(u),}
we define $R(t)$ for a fixed small $0<\e\ll\de/L$ 
\EQ{
 R(t):=\inf\{r>0\mid \exists\al\in\R^2,\ \int_{|x-\al|<r}e_F^{(1-c)}dx\ge E(u)-\e\}.}
Then for any $\al\in\R^2$ we have 
\EQ{
 \pt\int_{|x-\al|<R(t)}e_F^{(1-c)} dx \le E(u)-\e, 
 \pq \int_{|x-\al|>R(t)}e_F^{(1-c)} dx \le \e.}
The latter bound implies by Trudinger-Moser in the exterior disk, 
\EQ{
 \int_{|x-\al|>R(t)}[e_F+Df(u)+f(u)]dx \lec \e,}
provided $\e>0$ is small enough. 

By the same argument as in \cite{2DcNLKG}, we deduce that if 
\EQ{
 \liminf_{t\to\I}R(t)\ge 6,}
then for large $t$, $u$ is in the subcritical range for the Strichartz estimate decomposed into disks, so that we can obtain the scattering in the same way as in \cite{2DcNLKG}. Hence we may assume that there is $T_n\nearrow\I$ such that 
\EQ{
 R(T_n)< 6, \pq L \ll |T_n-T_{n+1}|\to\I.}
For each $n\in\N$, there is $\al_n\in\R^2$ such that 
\EQ{
 \int_{|x-\al_n|<6} e_F^{(1-c)} dx \ge E(u)-\e.}
Now we consider the virial identity localized to the cone 
\EQ{
 D_n:=\{(t,x)\in(T_n,T_{n+1})\times\R^2\mid |x-\al_n|<6+|t-T_n|\}.}
Let $\chi\in C_0^\I(\R^2)$ be a cut-off function satisfying 
\EQ{
 |x|\le 1\implies \chi(x)=1,\pq |x|\ge 2\implies\chi(x)=0.}
Let $w_n$ be a smooth cut-off for $D_n$, defined by 
\EQ{
 w_n(t,x)=\chi((x-\al_n)/(6+|t-T_n|)).}
Before using the virial identity, we need to estimate $|\al_n-\al_{n+1}|$ by using $P(u)=0$. Multiplying the equation with $(x-\al_n)w_n\dot u$, we get a localized center of energy 
\EQ{
 C_n(t):=\LR{(x-\al_n)w_n|e_N(u)}, \pq \dot C_n(t)=-P(u)+O(X_n(t)),} 
where $X_n(t)$ gathers the exterior energy terms 
\EQ{
  X_n(t):=\int_{(t,x)\not\in D_n}[e_F+Df(u)+f(u)]dx \pq(T_n \le t \le T_{n+1}).}
Since $X_n(T_n)\lec\e$, Lemma \ref{ext small} together with Trudinger-Moser implies $X_n(t)\lec\e$. 
For the boundary value at $t=T_n$ we have 
\EQ{
 |C_n(T_n)| \lec \|(x-\al_n)w_n(T_n)\|_\I \lec 1.}
Noting that $|\al_{n+1}-\al_n|\le |T_{n+1}-T_n|+O(1)$ by finite propagation speed, we have  
\EQ{
 \pt |C_n(T_{n+1})-(\al_{n+1}-\al_n)E(u)|
 \pr= \left|\int_{|x-\al_{n+1}|<6\ \cup\ |x-\al_{n+1}|>6}[(x-\al_n)w_n-(\al_{n+1}-\al_n)]e_N(u)dx\right|
 \pr\lec 1 + (\|(x-\al_n)w_n(T_{n+1})\|_\I+|\al_{n+1}-\al_n|)X_{n+1}(T_{n+1})
 \pr\lec 1 + \e(|T_{n+1}-T_n|+|\al_{n+1}-\al_n|).}
Thus we obtain 
\EQ{
 |\al_{n+1}-\al_n| \lec 1 + \e|T_{n+1}-T_n|.}

Next, the multiplier $2w_n[(x-\al_n)\cdot\na u+u]$ yields the localized virial identity 
\EQ{ \label{loc virial}
 V_n(t):=\LR{2w_n\dot u|(x-\al_n)\cdot\na u+u}, \pq \dot V_n=-K_2(u)+O(X_n),}
On the other hand, \eqref{kine lowbd} with Lemma \ref{K2 unif bd} implies that 
\EQ{
 \int_{T_n}^{T_{n+1}}K_2(u)dt \gec \frac{\de}{L}(T_{n+1}-T_n) \gg \e(T_{n+1}-T_n).}
So from \eqref{loc virial} we obtain 
\(
  V_n(T_{n+1})-V_n(T_n) \ll -\e(T_{n+1}-T_n), \)
contradicting 
\EQ{
 |T_n-T_{n+1}|\to\I,\ |V_n(T_n)|+|V_n(T_{n+1})|\lec 1+\e|T_{n+1}-T_n|,}
where the last estimate follows from the same argument as above for $C_n(T_{n+1})$. 

\subsubsection{Case 3: General momentum} 
Recall that the Lorentz transform $u(t,x)\mapsto w:=u(\al t+\be\cdot x,\al x+\be t)$ for any $(\al,\be)\in\mathbb{H}^d=\{(\al,\be)\in\R^{1+d}\mid \al=\sqrt{1-|s|^2}\}$ preserves global solutions, while 
\EQ{
 E(w)=\al E(u)+\be\cdot P(u),\pq P(w)=\al P(u)+ \be E(u).}
Hence if $u$ is a global solution with $E(u)>|P(u)|$, then there is $(\al,\be)\in\mathbb{H}^d$ such that $w$ is another global solution with $P(w)=0$ and $E(w)^2=E(u)^2-|P(u)|^2$. This is the case for any solution in $\K^+$, since 
\EQ{
 E(u) \ge E_F^{(1-c)}(u) \ge |\LR{\dot u|\na u}|,}
where the equalities hold only if $u\equiv 0$. 
Then by the result in the previous sections, $w$ scatters, i.e.~$\|\vec w-\vec z\|_2\to 0$ as $t\to\I$ for some free solution $z$. 
By the energy estimate, this asymptotic is valid also on space-like planes. More precisely, the pull back $v(t,x):=z(\al t-\be\cdot x,\al x -\be t)$ is a free solution such that $\|\vec u-\vec v\|_2\to 0$. 

To see this, cut off far exterior cones using Lemma \ref{ext small}, then the evolution of the energy of $w$ and $z$ on the space-like planes inside the light cones are controlled, via the energy identity, by the $L^1_tL^2_x$ norm of $f'(w)$ in some time slab, which is globally bounded and vanishing as $t\to\I$ by the scattering result for $w$. 

The above asymptotic implies also that $\|u(t)\|_2$ is bounded below for large time, hence $u$ is in the subcritical range and has global Strichartz norms. 

\section{Scattering in $\K^+$ in higher dimensions} 
\subsection{Global wellposedness in higher dimensions}\label{3Dglobal}
Suppose that the solution $u$ in $\K^+$ is defined for $-T<t<0$ with $E(u)=m$, blowing up as $t\to-0$. The local wellposedness argument by the Strichartz estimate implies that 
\EQ{
 \|u\|_{S(-T,0)}=\I}
for some appropriate Strichartz norm, say 
\( 
 S:=L^{2(d+1)/(d-2)}_{t,x}. \)

First we prove that the total energy has to concentrate inside a light cone. 
For any $0<\e\ll 1$ and $t\in(-T,0)$, we introduce the concentration radius at a fixed center $\al\in\R^d$, in a barely different definition from the 2D case 
\EQ{
 r_\e(t,\al):=\inf\{R>0 \mid \int_{|x-\al|>R}e_F(t)dx\le \e m\}}
and suppose that for some $t_0\in(-T,0)$, we have 
\EQ{ \label{far energy smallness}
 \log[r_\e(t_0,\al)/|t_0|] \gg 1/\e.}
Then by dyadic decomposition (in $|x-\al|$) of the Hardy and Sobolev inequalities, we can find some $R\in(2|t_0|,r_\e(t_0,\al)/2)$ such that 
\EQ{
 \int_{R<|x-\al|<2R}\frac{|u|^2}{|x-\al|^2}+|u|^{2^\star}dx \lec \frac{\|\na u\|_2^2}{\log[r_\e(t_0,\al)/|t_0|]} \ll \e m.}
Let $\chi\in C_0^\I(\R^d)$ satisfy $\chi(x)=1$ for $|x|\le 1$ and $\chi(x)=0$ for $|x|\ge 2$. Let $\chi_R(x):=\chi((x-\al)/R)$ and let $w$ be the solution of NLKG with 
\EQ{ 
 (w(t_0),\dot w(t_0))=\chi_R(u(t_0),\dot u(t_0)).}
Then we have 
\EQ{
 E(w) \pt\le \int_{|x-\al|<2R}e_N(t_0)dx-\int|u|^2\chi_R\De\chi_Rdx + \int_{R<|x-\al|<2R}\frac{|u|^{2^\star}}{2^\star}dx
 \pr\le E(u)-\int_{|x-\al|>R_\e(t_0)}e_N(t_0)dx + \frac{Cm}{\log[r_\e(t_0,\al)/|t_0|]}
 \pr\le (1-\e/2)m<m.}
Moreover, since 
\EQ{
 [\|\dot u\|_2^2+\|u\|_2^2]/2+\|\na u\|_2^2/d=E(u)-K_2(u)/2^\star\le m,}
we have 
\EQ{
 \int_{|x-\al|<r_\e(t_0,\al)}\frac{|\dot u|^2+|u|^2}{2}+\frac{|\na u|^2}{d}dx
 \le m(1-2\e/d),}
and so 
\( 
 \|\na w(t_0)\|_2^2 \le dm(1-\e/d)<dm=\|\na Q\|_2^2, \)
which implies that $K_2(w(t_0))\ge 0$. Hence by the result in \cite{ScatBlow}, $w$ is a global scattering solution. Moreover, the finite propagation speed of the linear equation implies that $w=u$ on $|x-\al|<2|t_0|-|t-t_0|$. Thus $u$ can be extended to $|x-\al|+t<|t_0|$. 

Therefore, if \eqref{far energy smallness} holds for some $\e>0$ and $t_0$ around each $\al\in\R^d$, then $u$ can be extended to some positive time. Note that an exterior cone is covered by Lemma \ref{ext small}. 
Hence to blow up at $t=0$, there must be $\al\in\R^d$ around which \eqref{far energy smallness} fails for all $\e$ and all $t_0$: For some constant $M>1$, any $0<\e\ll 1$ and any $t_0\in(-T,0)$, 
\EQ{
 \int_{|x-\al|>|t_0|M^{1/\e}}e_F(t_0)dx \le \e m.}
Hence by Lemma \ref{ext small}, we have for all $t\in(-T,0)$, 
\EQ{
 \int_{|x-\al|>|t_0|M^{1/\e}+|t-t_0|}e_F(t)dx \lec \e.}
Taking $t_0\to-0$ and then $\e\to+0$, we obtain 
\EQ{
 \supp u(t,x) \subset\{(t,x)\mid |x-\al|\le -t\},}
and so $\|u(t)\|_2\lec|t|$ by H\"older. 
Now suppose for contradiction that $M(t)\to 0$ as $t\to-0$. Then 
$K_{0}^{(c)Q}(u)=2E(u)+2F(u)-M(t)\ge 2m-o(1)$, so from \eqref{bd K}, 
\EQ{ 
 K_0^{(c)}(u)\ge\min(M(t),\de K_{0}^{(c)Q}(u))=M(t),}
for $t<0$ close to $0$, where 
\EQ{
 \ddot y/2=\|\dot u\|_2^2-K_0^{(c)}(u)-(1-c)\|u\|_2^2
 \le -2(1-c)\|u\|_2^2=-2(1-c)y.}
Hence $y(t)=\|u\|_2^2$ is concave near $t=0$, contradicting $\|u(t)\|_2^2\lec t^2$. Therefore $M(t)\not\to 0$, so there are $\de_0>0$ and $-T<t_n\nearrow 0$ such that 
\EQ{ \label{ut lowbd}
 \|\dot u(t_n)\|_2^2 \ge \de_0.}

Now apply the profile decomposition in \cite{ScatBlow} to the sequence of solutions 
\EQ{
 u_n(t):=u(t+t_n).} 
Since it is on the threshold and not scattering for $t>0$, we get exactly one profile, by the same proof as in \cite[\S 6]{ScatBlow}, except for a modification of Lemma 5.4 in an obvious way using \eqref{ut lowbd}, which ensures $J(u(t_n))\le m-\de_0/2$ is uniformly below the threshold. Thus we obtain, passing to a subsequence, 
\EQ{
 \vec u_n(0)=e^{-is_n\LR{\na}}\cT_n\fy+o(1) \pq\IN{L^2_x},}
where $\cT_n\fy:=h_n^{-d/2}\fy((x-\al_n)/h_n)$ for some sequences $s_n\in\R$, $h_n>0$, $\al_n\in\R^d$ and $\fy\in L^2$. Moreover, $h_n\to\exists h_\I\in\{0,1\}$ and $\ta_n:=-t_n/h_n\to\exists \ta_\I\in[-\I,\I]$ and the nonlinear profile $U_\I$ is defined as the solution of either the initial data problem or the final data problem given by, putting $\LR{\na}_\I:=\sqrt{-\De+h_\I^2}$, 
\EQ{
 \pt (\p_t^2-\De+h_\I)U_\I = f'(U_\I),
 \pq (e^{it\LR{\na}_\I}\fy-\vec U_\I)(\ta_n)\to 0 \IN{L^2}.}

If $U_\I$ is scattering both as $t\to\pm\I$ with finite Strichartz norm on $\R$, then it is a global approximation of $u_n$ with the scaling/translations, and so the long-time perturbation implies that $u_n$ is also a scattering solution for large $n$, contradicting the blowup of $u$. Hence $U_\I$ does not scatter both as $t\to\pm\I$. If $\ta_\I=\pm\I$, then similarly $u_n$ scatters as $t\to\pm\I$ with a uniform Strichart bound on $\pm t>0$, which is contradicting that $\|u\|_{S(-T,0)}=\I$. Hence $\ta_\I\in\R$. 
Thus we obtain 
\EQ{
 \cT_n^{-1}\vec u_n(0)=e^{i\ta_n\LR{\na}_n}\fy+o(1)=e^{i\ta_\I\LR{\na}_\I}\fy+o(1)=\vec U_\I(\ta_\I)+o(1).}
Hence $E^{(h_\I)}(U_\I)=m$, $K_2(U_\I(\ta_\I))\ge 0$, and \eqref{ut lowbd} implies 
\EQ{ \label{profile ut bd}
 \|\dot U_\I(\ta_\I)\|_2^2 \ge \de_0.}

 Since $\vec u_n(0)$ is concentrating as $n\to\I$, we must have $h_\I=0$, so the nonlinear profile $U_\I$ is a non-scattering solution of the massless NLW in $\K^+$ with $E=m$. 
According to Duyckaerts-Merle \cite{DM2}, such a solution is a scaling and translation of either $Q$ or $\pm W_-(\pm t)$, which is the global solution of NLW converging to $Q$ 
\EQ{
 \|W_--Q\|_{\dot H^1}+\|\dot W_-\|_2\to 0 \pq (t\to\I)}
and scattering for $t\to-\I$. 
Since \eqref{profile ut bd} precludes $Q$, we obtain 
\EQ{
 U_\I(t) = \s_1 R^{1-d/2}W_-(\s_2(t-T),(x-\al)/R).}
for some $\s_1,\s_2\in\{\pm 1\}$, $T\in\R$, $\al\in\R^d$ and $R>0$. 
Since $U_\I$ is scattering as $\s_2 t\to-\I$, it becomes a global approximation for $u_n$ with the scaling/translation on $\s_2 t<0$ for large $n$. Hence $u_n$ has uniformly bounded in the Strichartz norm on $\s_2 t<0$, contradicting $\|u\|_{S(-T,0)}=\I$. 

Thus we conclude that finite time blow-up is impossible. 

\subsection{Scattering in higher dimensions}
The argument is similar to the previous section. Let $u$ be a global solution in $\K^+$ with $E(u)=m$ and $\|u\|_{S(0,\I)}=\I$. First we claim that 
\EQ{
 \liminf_{t\to\I}\|\dot u\|_2>0.}
Suppose not. Then we have 
\EQ{
 0=\lim_{T\to\I}[\LR{u|\dot u}]_T^{T+1}=\lim_{T\to\I}\int_{T}^{T+1}K_0(u)dt.}
Since $K_0(u)=K_0^{(c)}(u)+(1-c)\|u\|_2^2$ and $\|u(t)\|_2$ is uniformly continuous, we deduce that $\|u\|_2\to 0$. Hence $M(t)\to 0$, and so by the same argument as in \S\ref{3Dglobal}, we obtain 
\EQ{
 \ddot y/2 \le -2(1-c)y}
for large $t$, contradicting $y\to 0$. 
Hence there are $\de_0>0$ and $t_n\nearrow\I$ such that 
\EQ{ \label{ut lowbd2}
 \|\dot u(t_n)\|_2^2 \ge \de_0.}

Now in the same way as in \S\ref{3Dglobal}, apply the profile decomposition to the sequence of solutions $u_n(t):=u(t+t_n)$, then we get exactly one profile. 

If $h_\I=0$ (i.e. the concentrating case) then the same argument implies that the profile is given by $W_-$ due to \eqref{ut lowbd2}, leading to a contradiction with $\|u\|_{S(0,\I)}=\I$. 
If $h_\I=1$, then the situation is essentially the same as in the subcritical case and so we can argue in the same way as in \cite{ScatBlow} to get a contradiction. 

\section{Acknowledgments}
The first author is thankful to Professor Yoshio Tsutsumi and all members of the Math Department at Kyoto University for their very generous hospitality.

\end{document}